\documentclass[12pt,leqno]{amsart}
\usepackage{verbatim,amssymb,amsfonts,latexsym,amsmath,amsthm,tikz}
\usetikzlibrary{arrows}

\newtheorem{theorem}{Theorem}[section]
\newtheorem*{maintheorem1}{Main Theorem (abridged)}
\newtheorem*{maintheorem2}{Main Theorem (unabridged)}
\newtheorem{lemma}[theorem]{Lemma}

\newtheorem{proposition}[theorem]{Proposition}
\newtheorem{question}[theorem]{Question}
\theoremstyle{definition}

\theoremstyle{remark}

\newtheorem*{sub-claim}{sub-claim}

\newcommand{\N}{\mathbb{N}}

\newcommand{\F}{\mathcal{F}}
\newcommand{\G}{\mathcal{G}}
\newcommand{\explicitSet}[1]{\left\lbrace #1 \right\rbrace}
\newcommand{\brackets}[1]{\left\langle #1 \right\rangle}
\newcommand{\set}[2]{\explicitSet{#1 \colon #2}}
\newcommand{\seq}[2]{\brackets{#1 \colon #2}}
\newcommand{\<}{\langle}
\renewcommand{\>}{\rangle}
\newcommand{\0}{\emptyset}
\renewcommand{\a}{\alpha}
\renewcommand{\b}{\beta}

\renewcommand{\k}{\kappa}

\newcommand{\w}{\omega}
\newcommand{\sub}{\subseteq}

\newcommand{\closure}[1]{\overline{#1}}

\newcommand{\card}[1]{\left\lvert #1 \right\rvert}

\renewcommand{\S}{\mathcal S}

\newcommand{\thick}{\Theta}

\newcommand{\subst}{\sub^{U}}
\renewcommand{\boxed}[1]{\text{\fboxsep=.2em\fbox{$#1$}}}
\newcommand{\gen}[1]{\<\!\< #1 \>\!\>}

\begin{document}

\title{Ideals and idempotents in the uniform ultrafilters}
\author{W. R. Brian}
\address {
William R. Brian\\
Department of Mathematics\\
Tulane University\\
6823 St. Charles Ave.\\
New Orleans, LA 70118}
\email{wbrian.math@gmail.com}

\maketitle

\begin{abstract}
If $S$ is a discrete semigroup, then $\b S$ has a natural, left-topological semigroup structure extending $S$. Under some very mild conditions, $U(S)$, the set of uniform ultrafilters on $S$, is a two-sided ideal of $\b S$, and therefore contains all of its minimal left ideals and minimal idempotents. We find some very general conditions under which $U(S)$ contains prime minimal left ideals and left-maximal idempotents.

If $S$ is countable, then $U(S) = S^*$, and a special case of our main theorem is that if a countable discrete semigroup $S$ is a weakly cancellative and left-cancellative, then $S^*$ contains prime minimal left ideals and left-maximal idempotents. We will provide examples of weakly cancellative semigroups where these conclusions fail, thus showing that this result is sharp.
\end{abstract}

\section{Introduction}

Semigroups of the form $\b S$ and $S^*$ have played a prominent part in modern combinatorics and algebra. Particularly important have been the minimal left ideals and the minimal idempotents, and these have become the object of a good deal of justifiable curiosity. Our main theorem concerns the existence of minimal left ideals and minimal idempotents with special topological and algebraic properties:

\begin{maintheorem1}
Let $S$ be a countable discrete semigroup. If $S$ is weakly cancellative, then
\begin{enumerate}
\item there is a minimal left ideal $L$ of $S^*$ that is also a weak $P$-set.
\end{enumerate}
If $S$ is also left-cancellative, then
\begin{enumerate}
\setcounter{enumi}{1}
\item $L$ is prime, i.e., $p \cdot q \in L$ if and only if $q \in L$.
\item the idempotents in $L$ are left-maximal (in particular, these idempotents are simultaneously minimal and maximal).
\end{enumerate}
\end{maintheorem1}

In the special case that $S$ is a group, $(2)$ and $(3)$ were proved by Zelenyuk in \cite{YZ2}. Before the appearance of Zelenyuk's paper, the existence of left-maximal idempotents in $S^*$, even for the most natural choices of $S$, was a longstanding open problem (see, e.g., Questions 9.25 and 9.26 in \cite{H&S}). Our proof uses a different technique from Zelenyuk's, and is based instead on a recent set-theoretic result of the author and Jonathan Verner in \cite{B&V}.

This ``abridged'' version of our main theorem solves the problem of left-maximal idempotents in $S^*$ in a very general setting. We will show by example that left cancellativity is necessary to prove $(2)$ and $(3)$, so that the result is fairly sharp.

The unabridged version of our main theorem concerns not $S^*$ but $U(S)$, the set of uniform ultrafilters on $S$. If $S$ is countable, then $U(S) = S^*$, but in the uncountable setting this is not the case. However, assuming $S$ is very weakly cancellative (defined in the next section), $U(S)$ is a two-sided ideal (hence a subsemigroup) of $\b S$. In particular, the minimal left ideals of $U(S)$ are precisely the minimal left ideals of $S^*$ and $\b S$.

\begin{maintheorem2}
Let $S$ be a discrete semigroup with $\card{S} = \k$, where $\k$ is regular. If $S$ is very weakly cancellative, then
\begin{enumerate}
\item there is a minimal left ideal $L$ that is a weak $P_{\k^+}$-set in $U(S)$.
\end{enumerate}
If $S$ is also left-cancellative, then
\begin{enumerate}
\setcounter{enumi}{1}
\item $L$ is prime in $U(S)$.
\item the idempotents in $L$ are left-maximal in $U(S)$ (in particular, these idempotents are simultaneously minimal and maximal).
\end{enumerate}
\end{maintheorem2}

In Section~\ref{sec:prelims}, we provide some background material concerning semifilters, ultrafilters, and the semigroups $\b S$, $S^*$, and $U(S)$. In Section~\ref{sec:countable}, we prove the abridged version of our main theorem and provide examples showing that the result is sharp. We have chosen to prove the abridged version first to highlight the main pattern of the proof, which is the same for the unabridged version. The proof of the abridged version hinges on a result from \cite{B&V} concerning filters and semifilters on countable sets. Before proving the unabridged version, we need to extend this result to uncountable sets of regular cardinality, which is done in Section~\ref{sec:lemma}. In Section~\ref{sec:general}, we prove the unabridged version of our main theorem.

\section{Definitions and Preliminaries}\label{sec:prelims}

\subsection*{Filters, semifilters, and ultrafilters}

A \emph{semifilter} on a set $S$ is a subset $\G$ of $\mathcal{P}(X)$ such that
\begin{itemize}
\item (nontriviality) $\0 \neq \G \neq \mathcal{P}(X)$.
\item (upward heredity) If $A \in \G$ and $A \sub^* B$ then $B \in \G$.
\end{itemize}
As usual, $A \sub^* B$ means that $A \setminus B$ is finite. $\G$ is a \emph{filter} if, additionally, $\G$ satisfies
\begin{itemize}
\item (downward directedness) If $A,B \in \G$ then $A \cap B \in \G$.
\end{itemize}
$\G$ is an \emph{ultrafilter} if $\G$ is a filter and
\begin{itemize}
\item (maximality) there is no filter properly extending $\G$.
\end{itemize}

Fix a semifilter $\S$. $\F$ is a \emph{filter on} $\S$ if $\F$ is a filter and $\F \sub \S$, and $\F$ is an \emph{ultrafilter on $\S$} if $\F$ is a maximal filter on $\S$: i.e., $\F$ is a filter on $\S$, but any filter properly containing $\F$ is not. Equivalently, $\F$ is an (ultra)filter on $\S$ if and only if $\F$ is an (ultra)filter on the partial order $(\S,\sub^*)$ (see \cite{B&V} for more on this).


The set of all ultrafilters on $S$ is denoted $\b S$. As usual, for $s \in S$ we identify the principal ultrafilter $\set{A \sub S}{s \in A}$ with $s$, so that $S \sub \b S$. We view $\b S$ as the \emph{Stone-\v{C}ech compactification} of $S$, where $S$ is assumed to have the discrete topology. The basic open neighborhoods of $\b S$ have the form $\closure{A} = \set{p \in \b S}{A \in p}$.

A \emph{free ultrafilter} on $S$ is an ultrafilter $p$ such that every $A \in p$ is infinite. The set of all free ultrafilters on $S$ is denoted $S^* = \b S \setminus S$. A \emph{uniform ultrafilter} on $S$ is an ultrafilter $p$ such that every $A \in p$ has the same cardinality as $S$. The set of uniform ultrafilters on $S$ is denoted $U(S)$.

Every filter $\F$ on $S$ corresponds to a closed subset of $\b S$, namely $\tilde \F = \bigcap_{A \in \F}\closure{A}$. Conversely, every closed $K \sub \b S$ is equal to $\tilde \F$ for some filter $\F$, namely $\F = \set{A \sub S}{K \sub \closure{A}}$. This correspondence is called \emph{Stone duality}, and $\tilde \F$ is called the \emph{Stone dual} of $\F$.

If $Y$ is a topological space and $X \sub Y$, then $X$ is a \emph{weak $P_\k$-set} (in $Y$) if whenever $D \sub Y \setminus X$ and $\card{D} < \k$, $\closure{D} \cap X = \0$. A weak $P_{\omega_1}$-set is just called a weak $P$-set. Roughly, a weak $P_\k$-set is a set that is far away from small enough subsets of its complement. A \emph{weak $P_\k$-point} is a point $p$ such that $\{p\}$ is a weak $P_\k$-set.

$X$ is a $P$-set if, for every countable collection $\set{U_n}{n \in \w}$ of neighborhoods of $X$, $X$ is in the interior of $\bigcap_{n \in \w}U_n$. In \cite{Brn}, it is proved for the case $S = (\N,+)$ to be independent of ZFC whether $S^*$ contains minimal left ideals that are also $P$-sets. Our abridged main theorem says that if we weaken ``$P$-set'' to ``weak $P$-set'' then we obtain a ZFC theorem rather than a consistency result. We note that this situation is exactly analogous to that for $P$-points and weak $P$-points in $\w^*$: the existence of $P$-points is independent of ZFC (Shelah, \cite{Shl}), but weak $P$-points always exist (Kunen, \cite{Kun}).

A variation of Kunen's proof was used in \cite{B&V} to obtain the following result. Recall that a subset of $\mathcal{P}(\w)$, such as a semifilter, can be identified with a subset of $2^\w$ via characteristic functions, so we can talk about semifilters as being comeager, $G_\delta$, etc.

\begin{proposition}\label{prop:weakPsets}
If $\G$ is a comeager semifilter, then there is an ultrafilter $\F$ on $\G$ such that $\tilde \F$ is a weak $P$-set in $\w^*$.
\end{proposition}
\begin{proof}
See Section 5 of \cite{B&V}. Alternatively, see our generalization of this result in Theorem~\ref{thm:lemma} below.
\end{proof}

This proposition will be crucial to the proof of the abridged version of our main theorem. For the unabridged version, we will use a generalization of this result, given as Theorem~\ref{thm:lemma} below, that applies to semifilters on uncountable sets of regular cardinality.

\subsection*{$\b\w$ as a semigroup}

Henceforth, $S$ denotes a semigroup with the discrete topology and with $\cdot$ as its binary operation.

There is a unique extension of the operator $\cdot$ to all of $\b S$ such that
\begin{itemize}
\item for each $p \in \b S$, the function $x \mapsto x \cdot p$ is continuous on $\b S$.
\item for each $s \in S$, the function $x \mapsto s \cdot x$ is continuous on $\b S$.
\end{itemize}
This extension is accomplished in the natural way by applying the Stone extension property of $\b S$ twice in a row (see Theorem 4.1 of \cite{H&S} for details). We can also define this operation explicitly:
$$p \cdot q = \set{A \sub S}{\set{s}{s^{-1} A \in q} \in p}.$$
Here, as below, $s^{-1} A = \set{t \in S}{s \cdot t \in A}$. Notice that this definition makes sense even when $s$ does not have a left inverse in $S$.

The semigroup $(\b S,\cdot)$ is usually abbreviated $\b S$. If $S$ is sufficiently well-behaved (see Proposition~\ref{prop:fact3} below), then $S^*$ and $U(S)$ are subsemigroups of $\b S$.

A \emph{left ideal} of $\b S$ is a nonempty $L \sub \b S$ such that $\b S \cdot L \sub L$. A \emph{minimal left ideal} is a left ideal that does not properly contain any other left ideal. $L$ is a \emph{prime} left ideal if $p \cdot q \notin L$ whenever $q \notin L$.

\begin{proposition}\label{prop:fact2}
If $L$ is closed in $\b S$, then $L$ is a left ideal of $\b S$ if and only if, for every $s \in S$, $s \cdot L \sub L$.
\end{proposition}
\begin{proof}
See, e.g., Lemma 19.4 of \cite{H&S}.
\end{proof}

An \emph{idempotent ultrafilter} is any $p \in \b S$ such that $p \cdot p = p$. The idempotents of $\b S$ admit three natural partial orders:
$$p \leq_L q \qquad \Leftrightarrow \qquad p \cdot q = p,$$
$$p \leq_R q \qquad \Leftrightarrow \qquad q \cdot p = p,$$
$$p \leq q \qquad \Leftrightarrow \qquad p \leq_L q \ \text{ and } \ p \leq_R q.$$

An idempotent $p$ is minimal with respect to one of these orders if and only if it is minimal with respect to all three, and this is true if and only if $p$ belongs to a minimal left ideal (see, e.g., Theorems 1.36 and 1.38 in \cite{H&S}). Idempotents of this kind are simply called \emph{minimal idempotents}. An idempotent is called \emph{left-maximal}, \emph{right-maximal}, or \emph{maximal} if it is maximal with respec to $\leq_L$, $\leq_R$, or $\leq$, respectively.

Using compactness, it is fairly easy to show that $\b S$, $S^*$, and $U(S)$ contain right-maximal idempotents (see, e.g., Theorem 2.12 in \cite{H&S}). However, the existence of left-maximal idempotents has been a difficult problem. Even for $S = (\mathbb Z,+)$, the existence of left-maximal idempotents in $\b S$ was open for years, despite the fact that much work was done on this problem (see Questions 9.25 and 9.26 in \cite{H&S}; see also Questions 5.5(2),(3) in \cite{HSZ}, Problems 4.6 and 4.7 in \cite{Brn}, and \cite{Zln}). The following Proposition (a known folklore result) tells us one way to find left-maximal idempotents in $\b S$.

\begin{proposition}\label{prop:2implies3}
If $\b S$ (or $S^*$ or $U(S)$) contains a minimal left ideal that is also prime, then it contains an idempotent that is both minimal and left-maximal.
\end{proposition}
\begin{proof}
If $L$ is a minimal left ideal, then $L$ contains a minimal idempotent $p$ (this is sometimes called ``Ellis's Theorem'' or the ``Ellis-Numakura Lemma''; see Corollary 2.6 in \cite{H&S}). If $L$ is also a prime ideal, we claim that $p$ is also $\leq_L$-maximal. If $q$ is an idempotent and $q \neq p$, then either $q \in L$, in which case $p \not\leq_L q$ because $q$ is also minimal, or $q \notin L$, in which case $q+p \notin L$ (because $L$ is prime), so in particular $p+q \neq p$ and $p \not\leq q$. In either case, $p \not\leq q$ and, as $q$ was arbitrary, $p$ is $\leq_L$-maximal.
\end{proof}

For a fixed $s \in S$, let $\lambda_s(x) = s \cdot x$ and $\rho_s(x) = x \cdot s$. We can classify semigroups according to how near to being injective these functions are:
\begin{itemize}
\item $S$ is \emph{left cancellative} if for every $s,t \in S$, $\card{\lambda^{-1}_s(t)} = 1$ (i.e., $\lambda_s$ is one-to-one).
\item $S$ is \emph{$\k$-weakly cancellative} if for every $s,t \in S$, $\card{\lambda^{-1}_s(t)} < \k$ and $\card{\rho^{-1}_s(t)} < \k$ (i.e., $\lambda_s$ and $\rho_s$ are ($<\!\k$)-to-one).
\item $S$ is \emph{weakly cancellative} if it is $\aleph_0$-weakly cancellative.
\item $S$ is \emph{very weakly cancellative} if it is $\card{S}$-weakly cancellative.
\end{itemize}

One could similarly define right cancellative, right and left $\k$-weakly cancellative, etc., but the definitions we have stated are the only ones we will need. Notice that, as with $S^*$ and $U(S)$, the distinction between weak and very weak cancellativity collapses if $S$ is countable.

$R$ is a \emph{right ideal} of $\b S$ if $R \cdot \b S \sub R$, and $I$ is a \emph{two-sided ideal} if it is both a left and right ideal. Recall that the union of all the minimal left ideals of $\b S$ is a two-sided ideal, and is contained in every other two-sided ideal of $\b S$.

\begin{proposition}\label{prop:fact3}
Let $S$ be a semigroup of size $\k$.
\begin{enumerate}
\item If $S$ is weakly cancellative, then $S^*$ is a subsemigroup of $\b S$. Moreover, $S^*$ is a two-sided ideal, so every minimal left ideal of $\b S$ is contained in $S^*$.
\item If $\k$ is regular and $S$ is very weakly cancellative, then $U(S)$ is a subsemigroup of $\b S$. Moreover, $U(S)$ is a two-sided ideal, so every minimal left ideal of $\b S$ is contained in $U(S)$.
\item If $\k$ is singular and $S$ is $\mu$-weakly cancellative for some $\mu < \k$, then $U(S)$ is a subsemigroup of $\b S$. Moreover, $U(S)$ is a two-sided ideal, so every minimal left ideal of $\b S$ is contained in $U(S)$.
\end{enumerate}
\end{proposition}
\begin{proof}
$(1)$ is a special case of $(2)$. This case is already known: see Theorems 4.36 and 4.37 in \cite{H&S}.

The arguments for $(2)$ and $(3)$ are the same except at one or two points. We will prove them simultaneously, and consider the two different cases only at the points where we need to.

First we show that $U(S)$ is a left ideal. Let $B \sub S$ with $\card{B} < \k$, and fix $s \in S$. If $\k$ is regular (respectively, singular), then very weak cancellativity (respectively, $\mu$-weak cancellativity) implies
$$\card{s^{-1}B} = \card{\bigcup_{t \in B}\lambda_s^{-1}(t)} < \k.$$
If $q \in U(S)$, it follows that $\set{s}{s^{-1}B \in q} = \0$. Therefore
$$B \notin \set{A}{\set{s}{s^{-1} A \in q} \in p} = p \cdot q$$
for any $p \in \b S$. Since $B$ was arbitrary with $\card{B} < \k$, this shows $p \cdot q \in U(S)$ for every $p \in \b S$. Thus $U(S)$ is a left ideal.

Next we show that $U(S)$ is a right ideal. Fix $p \in U(S)$ and $q \in \b S$; we must show $p \cdot q \in U(S)$. If $q \in U(S)$ then this follows from the previous paragraph. So assume $q \notin U(S)$, and fix $A \sub S$ with $\card{A} < \k$ and $A \in q$.

If $C \sub S$ and $\card{C} < \k$, consider
$$X_C = \set{s \in S}{(s \cdot C) \cap C \neq \0}.$$
Suppose $\card{X_C} = \k$. To every $s \in X_C$ we can associate a pair $p(s) = (c_1,c_2) \in C \times C$ such that $s \cdot c_1 = c_2$. If $\k$ is regular, by the pigeonhole principle there must be some $c_1,c_2$ such that $p(s) = (c_1,c_2)$ for $\k$ distinct values of $s$, contradicting very weak cancellativity. Similarly, if $\k$ is singular then there must be some $c_1,c_2$ such that $p(s) = (c_1,c_2)$ for at least $\mu$ distinct values of $s$, contradicting $\mu$-weak cancellativity. In either case we have a contradiction, so $\card{X_C} < \k$.

Now suppose $p \cdot q \notin U(S)$, and fix $B \in p \cdot q$ with $\card{B} < \k$. Letting $C = A \cup B$, we have $C \in q$, $C \in p \cdot q$, and $\card{C} < \k$.
\begin{align*}
& \set{s \in S}{s^{-1}C \in q} = \set{s \in S}{s^{-1}C \cap C \in q} \\
& \qquad \sub \set{s \in S}{s^{-1}C \cap C \neq \0} = \set{s \in S}{(s \cdot C) \cap C \neq \0} = X_C,
\end{align*}
so $\card{\set{s \in S}{s^{-1}C \in q}} < \k$ by the previous paragraph. Since $p \in U(S)$ and $p \cdot q = \set{C \sub S}{\set{s}{s^{-1} C \in q} \in p}$, this is a contradiction. Thus $p \cdot q \in U(S)$, which completes the proof that $U(S)$ is a right ideal.

This shows $U(S)$ is a two-sided ideal under the assumptions in $(2)$ or $(3)$. Every two-sided ideal in $\b S$ is a subsemigroup of $\b S$ and contains all the minimal left ideals, so we are done.
\end{proof}

In the remainder of the paper, we will work mostly with semigroups satisfying the assumptions of $(2)$ or $(3)$. We will use Proposition~\ref{prop:fact3} implicitly in several places.

\section{The main theorem (abridged)}\label{sec:countable}

In this section we prove the abridged version of our main theorem through a sequence of smaller propositions and lemmas. Some of these results are valid for uncountable semigroups and will be used again in Section~\ref{sec:general} to prove the unabridged version. We will explicitly mark every result that depends on the countability of $S$.

Say that $A \sub S$ is $S$\emph{-thick} whenever, for every finite $F \sub S$, there is some $s \in S$ such that $F \cdot s \sub A$ (this definition is taken from \cite{BHM}). We denote the collection of $S$-thick sets by $\thick_S$.

\begin{proposition}\label{prop:fact1}
Suppose $S$ is any semigroup and $A \sub S$. $A \in \thick_S$ if and only if there is some (minimal) left ideal $L$ of $\b S$ such that $L \sub \closure{A}$.
\end{proposition}
\begin{proof}
See Theorem 2.9(c) in \cite{BHM}.
\end{proof}

Proposition~\ref{prop:fact1} connects the $S$-thick sets to the algebra of $\b S$, and the next proposition does so in a much stronger way. The special case $S = (\N,+)$ is proved in \cite{Brn} (see Lemma 3.2).

\begin{proposition}\label{prop:thicksets}
Suppose $S$ is any semigroup and $\F$ is a filter on $S$. Then $\tilde \F$ is a minimal left ideal in $\b S$ if and only if $\F$ is an ultrafilter on $\thick_S$.
\end{proposition}

\begin{lemma}\label{lem:thickshifts}
Let $S$ be any semigroup. If $A \sub S$ is $S$-thick, then, for any $s \in S$, $A \cap s^{-1}A$ is also $S$-thick.
\end{lemma}
\begin{proof}
Fix $A \in \thick_S$, $s \in S$, and let $F$ be any finite subset of $S$. $F \cup (s \cdot F)$ is also finite, so there is some $t \in S$ such that $(F \cup (n \cdot F)) \cdot t = (F \cdot t) \cup (s \cdot F \cdot t) \sub A$. Then $F \cdot t \sub A$ and $s \cdot F \cdot t \sub A$. The latter implies $F \cdot t \sub s^{-1}A$, so we get $F \cdot t \sub A \cap s^{-1}A$. Since $F$ was arbitrary, $A \cap s^{-1}A$ is $S$-thick.
\end{proof}

\begin{lemma}\label{lem:inverseimages}
Let $S$ be any semigroup and $s \in S$. If $\F$ is an ultrafilter on $\thick_S$ and if $A \in \F$, then $s^{-1}A \in \F$.
\end{lemma}
\begin{proof}
Let $\F$ be an ultrafilter on $\thick_S$ and let $A \in \F$, $s \in S$. Let $B \in \F$. Then $A \cap B \in \F \sub \thick_S$, and by Lemma~\ref{lem:thickshifts} we have
$$(A \cap B) \cap s^{-1}(A \cap B) \in \thick_S.$$
Since $\thick_S$ is closed upwards under $\sub$ and
$$B \cap s^{-1}A \supseteq (A \cap B) \cap s^{-1}(A \cap B),$$
we have $B \cap s^{-1}A \in \thick_S$. Since $B$ was an arbitrary member of $\F$, this means that the filter $\F'$ generated by $\F \cup \{s^{-1}A\}$ is a filter on $\thick_S$. But $\F$ is an ultrafilter on $\thick_S$, and clearly $\F \sub \F'$. Thus $\F = \F'$, which gives $s^{-1}A \in \F$.
\end{proof}

\begin{proof}[Proof of Proposition~\ref{prop:thicksets}]
First, recall that every minimal left ideal $L$ of $\b S$ is closed. By Stone duality, if $L$ is a closed left ideal then there is some filter $\F$ with $L = \tilde \F$. Therefore, $\tilde \F$ is a minimal left ideal if and only if $\tilde \F$ is a left ideal, and there is no filter $\G$ properly extending $\F$ such that $\tilde \G$ is also a left ideal.

By Proposition~\ref{prop:fact1}, $A$ is $S$-thick if and only if $\closure{A}$ contains a (minimal) left ideal. Thus, if $\tilde \F$ is a left ideal then $\F \sub \thick_S$; i.e., $\F$ is a filter on $\thick_S$. Given this fact and the argument of the previous paragraph, it suffices to show that if $\F$ is an ultrafilter on $\thick_S$ then $\tilde \F$ is a left ideal.

By Proposition~\ref{prop:fact2}, it suffices to show that if $\F$ is an ultrafilter on $\thick_S$ then for every $s \in S$ we have $s \cdot \tilde \F \sub \tilde \F$. Indeed,
\begin{align*}
s \cdot \tilde \F = s \cdot \bigcap_{A \in \F}\closure{A} & \sub \bigcap_{A \in \F}s \cdot \closure{A} = \bigcap_{A \in \F}\closure{s \cdot A} \\
& \sub \bigcap_{A \in \F}\closure{s \cdot (s^{-1}A)} \sub \bigcap_{A \in \F}\closure{A} = \tilde \F.
\end{align*}
The first and last equalities are just the definition of $\tilde \F$. The middle equality it true because, by the continuity of $x \mapsto s \cdot x$, we have $s \cdot \closure{A} = \closure{s \cdot A}$ for every $A \sub S$. The first inclusion is obvious, the second is true by Lemma~\ref{lem:inverseimages}, and the third follows from the fact that $s \cdot s^{-1}A \sub A$ for every $A \sub S$ (which is clear from the definition of $s^{-1}A$).
\end{proof}

\begin{lemma}\label{lem:itsasemifilter}
Let $S$ be a countable semigroup. If $S$ is weakly cancellative, then $\thick_S$ is a comeager semifilter.
\end{lemma}
\begin{proof}
Clearly $\0 \neq \thick_S \neq \mathcal{P}(S)$. If $A \sub^* B$ and $A \in \thick_S$, it follows immediately from Proposition~\ref{prop:fact1} and Proposition~\ref{prop:fact3}(1) that $B \in \thick_S$. Thus $\thick_S$ is a semifilter. To prove that it is comeager, we first notice that $\thick_S$ is dense in $2^S$ (for example, because $\thick_S$ contains every co-finite set, and the set of all co-finite sets is dense in $2^S$). Therefore it suffices to show that $\thick_S$ is $G_\delta$ in $2^S$. Let
$$U_F^t = \set{X \sub S}{F \cdot t \sub X} \text{ and }$$
$$U_F = \set{X \sub S}{\exists t \in S (F \cdot t \sub X)}.$$
Each $U_F^t$ is a basic open set in $2^S$, so each $U_F = \bigcup_{t \in S}U_F^t$ is also open. By definition, $\thick_S = \bigcap \set{U_F}{F \sub S, \card{F} < \aleph_0}$.
\end{proof}

Putting these pieces together, we obtain a proof of part $(1)$ of the abridged version of our main theorem:

\begin{theorem}\label{thm:Psets}
If $S$ is weakly cancellative, then there is a minimal left ideal of $\b S$ that is a weak $P$-set in $S^*$.
\end{theorem}
\begin{proof}
This follows immediately from Proposition~\ref{prop:weakPsets}, Proposition~\ref{prop:thicksets}, and Lemma~\ref{lem:itsasemifilter}.
\end{proof}

As mentioned in the introduction, left cancellativity is necessary for proving parts $(2)$ and $(3)$ of our main theorem. Notice that we have not yet used this assumption, and the next lemma is the only place where it is used in our argument. Thus parts $(2)$ and $(3)$ of our theorem are true for any weakly cancellative semigroup satisfying the conclusion of the following lemma.

\begin{lemma}\label{lem:exercise}
Suppose $S$ is left cancellative, let $L$ be a minimal left ideal of $\b S$, and $p \notin L$. For every $s \in S$, $s \cdot p \notin L$.
\end{lemma}
\begin{proof}
This is a slight modification of Exercise 8.2.2(ii) in \cite{H&S}.

Since $S$ is left cancellative, the function $x \mapsto s \cdot x$ is injective on $S$. It follows that the function $s \mapsto s \cdot x$ is also injective on $\b S$. Fix a minimal left ideal $L$, and suppose $s \cdot p \in L$. By the Structure Theorem (see, e.g., Theorem 1.64 in \cite{H&S}), $s \cdot p$ is a member of some group contained in $L$. In particular, there is some $e \in L$ such that $s \cdot p \cdot e = s \cdot p$. But we have already said that the function $x \mapsto s \cdot x$ is injective, so $p = p \cdot e$. Since $e \in L$ and $L$ is a left ideal, $p \cdot e \in L$.
\end{proof}

We can now prove part $(2)$ of the abridged version of our main theorem:

\begin{theorem}\label{thm:betaS}
Suppose $S$ is left-cancellative and weakly right cancellative, and that $L$ is a minimal left ideal of $S^*$ that is also a weak $P$-set. Then $L$ is prime.
\end{theorem}
\begin{proof}
Suppose $L$ is both a weak $P$-set and a minimal left ideal. By Lemma~\ref{lem:exercise}, if $p \notin L$ then $(S \cdot p) \cap L = \0$. Since $L$ is a weak $P$-set, $\b S \cdot p = \closure S^{\b S} \cdot p = \closure{S \cdot p}^{\b S}$ is disjoint from $L$ (the last equality follows from the continuity of the map $x \mapsto x \cdot p$). Therefore, if $p \notin L$ then $q \cdot p \notin L$. In other words, $L$ is prime.
\end{proof}

We showed already (in Proposition~\ref{prop:2implies3}) that part $(2)$ of our main theorem implies part $(3)$, so this concludes the proof of the abridged version of our main theorem.

Before moving on to the unabridged version, we will give some examples that show left cancellativity is necessary for parts $(2)$ and $(3)$.

Let $(F,\star)$ be any finite semigroup, and let $S$ be the product of $(\w,\max)$ with $(F,\star)$. Explicitly, $S = (\w \times F,\cdot)$, where
$$(m,a) \cdot (n,b) = (\max(m,n),a \star b).$$
This is easily seen to be a weakly cancellative semigroup.

Because $F$ is finite, there is a natural identification of the ultrafilters on $\w \times F$ with $\b\w \times F$. To see this, let $\F$ be an ultrafilter on $\w \times F$ and note that, since $F$ is finite, $\w \times \{a\} \in \F$ for exactly one $a \in F$. If $p = \set{A \sub \w}{A \times \{a\} \in \F}$, then $p \in \b\w$, and we identify $\F$ with the pair $(p,a)$. Using this notation, we have $\b S = (\b\w \times F,\cdot)$.

For each $a \in F$, define the function $\Lambda_a$ on $\b\w \times F$ by
$$\Lambda_a(p,x) = (p,a \star x).$$
Note that $\Lambda_a$ is continuous. For $a \in F$ and $q \in \b\w$, define
$$P_a^q(p,x) = (q,x \star a),$$
and note that this function is also continuous.

For any infinite $A \sub \w \times F$ and any $(m,a) \in \w \times F$,
\begin{align*}
(m,a) \cdot A & = \set{(\max(m,n),a \star b)}{(n,b) \in A} \\
& =^* \set{(n,a \star b)}{(n,b) \in A} = \Lambda_a(A),
\end{align*}
where, as usual, $X =^* Y$ means that $X \sub^*Y$ and $Y \sub^* X$. Since this is true for every infinite $A$, the function $x \mapsto (m,a) \cdot x$ must be equal to $\Lambda_a$ on $S^*$. That is, $(m,a) \cdot (q,b) = \Lambda_a(q,b) = (q,a \star b)$ for every $q \in \w^*$. Fixing $(q,b)$, we have showed that the function $x \mapsto x \cdot (q,b)$ is equal to $P_b^q$ on $\w \times F$. Since $x \mapsto x \cdot (q,b)$ must be continuous on all of $\b\w \times F$, it must be equal to $P_b^q$ everywhere. Explicitly,
$$(p,a) \cdot (q,b) = P_b^q(p,a) = (q,a \star b)$$
for every $p,q \in \w^*$ and $a,b \in F$. We now have a complete description of the semigroup $\b S$.

\begin{proposition}
There is a countable weakly cancellative semigroup $S$ such that no minimal left ideal of $\b S$ is prime and no minimal idempotent is maximal.
\end{proposition}
\begin{proof}
Let $F = \{0,1\}$, let $\star$ denote multiplication, and let $S$ be the semigroup described above.

Every element of $S^*$ is idempotent, since $(p,0) \cdot (p,0) = (p,0)$ and $(p,1) \cdot (p,1) = (p,1)$. Notice that $(p,0) \cdot (p,1) = (p,1) \cdot (p,0) = (p,0)$, so $(p,0) \leq (p,1)$ for every $p \in \w^*$. Thus nothing of the form $(p,0)$ is maximal, and nothing of the form $(p,1)$ is minimal.

If $p \in \w^*$ and $(q,a) \in \w^* \times \{0,1\}$, then $(q,a) \cdot (p,0) = (p,0)$. Thus $\{(p,0)\}$ is a minimal left ideal. By the previous paragraph, every minimal left ideal has this form. However, for any $q \in \w^*$ we have $(q,0) \cdot (p,1) = (p,0)$, so that $\{(p,0)\}$ is not prime.
\end{proof}

This shows that parts $(2)$ and $(3)$ of our main theorem can both fail for weakly cancellative semigroups. It can also happen that $(2)$ fails and $(3)$ holds (though Proposition~\ref{prop:2implies3} forbids the reverse situation).

\begin{proposition}
There is a countable weakly cancellative semigroup $S$ such that no minimal left ideal of $\b S$ is prime, but $\b S$ still contains idempotents that are both minimal and left-maximal.
\end{proposition}
\begin{proof}
Let $F = \{0,1\}$, let $\star$ be the trivial binary operator that maps every pair to $0$, and let $S$ be the semigroup described above.

The idempotents of $S^*$ are precisely the points of the form $(p,0)$. As in the proof of the previous proposition, the minimal left ideals are precisely the sets of the form $\{(p,0)\}$. Thus every idempotent is minimal, left-maximal, and right-maximal. However, no minimal left ideal is prime, since $(q,0) \cdot (p,1) = (q,1) \cdot (p,1) = (p,0)$ for every $q$.
\end{proof}

\section{A more general lemma}\label{sec:lemma}

The proof of the unabridged version of our main theorem is essentially the same as that of the abridged version, once we have obtained the appropriate generalization of Proposition~\ref{prop:weakPsets}.

For the rest of this section, $S$ is a set (its semigroup structure is irrelevant) with $\card{S} = \k$. For $A,B \sub S$, $A \subst B$ means $\card{A \setminus B} < \k$. A semifilter $\G$ on $S$ is \emph{large} if
\begin{itemize}
\item If $A \in \G$ and $A \subst B$, then $B \in \G$.
\item there is a partition $\set{X_\a}{\a < \k}$ of $S$ such that $\card{X_\a} < \k$ for all $\a$, and if $A \in [\k]^{\k}$, then $\bigcup_{\a \in A}X_\a \in \G$.
\end{itemize}
By a result of Talagrand (see, e.g., Proposition 2.2 in \cite{B&V}), if $S$ is countable then a semifilter $\G$ on $S$ is large if and only if it is comeager. The main result of this section is the following extension of Proposition~\ref{prop:weakPsets}:

\begin{theorem}\label{thm:lemma}
Suppose $\k$ is regular and $\G$ is a large semifilter on $S$. There is an ultrafilter $\F$ on $\G$ such that $\tilde \F$ is a weak $P_{\k^+}$-set in $U(S)$.
\end{theorem}

For the case $\G = [\k]^{\k}$, this theorem asserts that there is some $p \in U(S)$ that is a weak $P_{\k^+}$-point in $U(S)$. This result was proved by Baker and Kunen in \cite{B&K}. Our proof of Theorem~\ref{thm:lemma} is a modification of their proof. The changes we make are essentially technical: that is, the essential idea of the proof remains mostly the same. Our exposition here will be somewhat terse, since our main goal is just to show how to modify the proof in \cite{B&K}, and that the proof still works with these modifications in place.

We begin by recalling some relevant definitions:
\begin{itemize}
\item $\mathcal{FR}$ is the filter $\set{A \sub S}{\card{S \setminus A} < \k}$ (note $U(S) = \tilde{\mathcal{FR}}$).
\item Given a sequence $\seq{X_\a}{\a < \k}$ of subsets of $S$ and $p \in [\k]^{<\w}$, let $X_{\boxed{p}} = \bigcap_{\a \in p}X_\a$.
\item A \emph{hatfunction} is a function $\widehat{\ }: [\k^+]^{<\w} \to [\k]^{<\w}$ such that $\widehat \0 = \0$ and if $p \sub q$ then $\widehat p \sub \widehat q$.
\item Given a hatfunction $\widehat{\ \ }$ and a closed $X \sub U(S)$, $X$ is a $\widehat{\ \ }$\emph{set} in $U(S)$ if, for any monotone collection $\set{U_r}{r \in [\k]^{<\w}}$ of neighborhoods of $X$ (in this context, \emph{monotone} means that $U_r \supseteq U_s$ whenever $r \sub s$), there are neighborhoods $\set{V_\a}{\a < \k^+}$ of $X$ such that $V_{\boxed{r}} \sub U_{\widehat{r}}$ for every $r \in [\k^+]^{<\w}$.
\end{itemize}

\begin{lemma}\label{lem:hatset}
There is a hatfunction $\widehat{\ }: [\k^+]^{<\w} \to [\k]^{<\w}$ such that every $\widehat{\ \ }$set $X$ is also a weak $P_{\k^+}$-set in $U(S)$.
\end{lemma}
\begin{proof}
See Lemmas 2.5, 2.6, and 5.2 in \cite{B&K} for the case when $X$ is a point. The extension to arbitrary $\widehat{\ \ }$sets is trivial; some discussion of this can be found in Section 2 of \cite{BK2}.
\end{proof}

Although we do not need to define these terms here, we point out that the $\widehat{\ \ }$set mentioned in this lemma is both $\k^+$-good and $\k^+$-mediocre (which is stronger than merely being a weak $P_{\k^+}$-set).

Given the hatfunction mentioned in Lemma~\ref{lem:hatset}, the proof of Theorem~\ref{thm:lemma} is accomplished by constructing an ultrafilter $\F$ on $\G$ such that $\tilde \F$ is a $\widehat{\ \ }$set in $U(S)$. The construction is by a $2^\k$-step transfinite recursion using a suitably chosen matrix of sets in $\G$.

\begin{itemize}
\item Given a hatfunction $\widehat{\ \ }$ and a large semifilter $\G$, $\mathcal M$ is a $\widehat{\ \ }$\emph{step-family} over $\G$ if $\mathcal M = \set{E_r}{r \in [\k]^{<\w}} \cup \set{A_\a}{\a < \k^+}$ and
\begin{enumerate}
\item $E_s \cap E_t = \0$ for distinct $s,t \in [\k]^{<\w}$.
\item $\card{A_{\boxed{p}} \cap \bigcup \set{E_s}{s \not\supseteq \widehat p}} < \k$ for each $p \in [\k^+]^{<\w}$.
\item $\widehat p \sub s$ implies $A_{\boxed{p}} \cap E_s \in \G$ for each $p \in [\k^+]^{<\w}$, $s \in [\k]^{<\w}$.
\end{enumerate}
\item Given a filter $\F$ on $\G$ and a hatfunction $\widehat{\ \ }$, an \emph{$(\F,\G)$-independent $\widehat{\ \ }$step-family matrix over $I$} is a collection $\set{\mathcal M^i}{i \in I}$ with
$$\mathcal M^i = \set{E^i_r}{r \in [\k]^{<\w}} \cup \set{A^i_\a}{\a < \k^+}$$
such that
\begin{enumerate}
\item For each $i \in I$, $\mathcal M^i$ is a $\widehat{\ \ }$step-family over $\G$.
\item If $p_0,\dots,p_n \in [\k^+]^{<\w}$, $s_0,\dots,s_n \in [\k]^{<\w}$ with each $\widehat p_k \sub s_k$, and $i_0,\dots,i_n \in I$ with all of the $i_k$ distinct, then
$$\left( A^{i_0}_{\boxed{p_0}} \cap E^{i_0}_{s_0} \right) \cap \dots \cap \left( A^{i_n}_{\boxed{p_n}} \cap E^{i_n}_{s_n} \right) \cap C \in \G,$$
for every $C \in \F$.
\end{enumerate}
\end{itemize}

\begin{lemma}\label{lem:matrix}
Suppose $\k$ is regular, $\widehat{\ \ }$ is a hatfunction, and $\G$ is a large semifilter on $S$. There is an $(\mathcal{FR},\G)$-independent $\widehat{\ \ }$step-family matrix over $2^\k$.
\end{lemma}
\begin{proof}
The case $\G = [\k]^{\k}$ was proved (for regular $\k$ only) in \cite{B&K} as Theorem 4.5. Let $\set{\mathcal M^i_0}{i \in 2^\k}$ be a matrix satisfying the conclusions of the lemma for the case $\G = [\k]^{\k}$.

Using the fact that $\G$ is a large semifilter, find a partition $\set{X_\a}{\a < \k}$ of $S$ such that $\card{X_\a} < \k$ for each $\a$ and if $A \in [\k]^{\k}$ then $\bigcup_{\a \in A}X_\a \in \G$. For every $B \sub \k$, let $B' = \bigcup_{\a \in B}X_\a$. For each $i \in 2^\k$ let $\mathcal M^i = \set{B'}{B \in \mathcal M^i_0}$. It is routine to check that $\set{\mathcal M^i}{i \in 2^\k}$ is as required. Note that checking part $(2)$ of the definition of a $\widehat{\ \ }$step-family uses the regularity of $\k$.
\end{proof}

Lemma~\ref{lem:matrix} is the only place in this section where the regularity of $\k$ is needed. Thus a positive answer to the following question would mean that Theorem~\ref{thm:lemma} holds for singular $\k$ too.

\begin{question}\label{question:matrix}
Is Lemma~\ref{lem:matrix} true for singular $\k$?
\end{question}

As we will see in the next section, if Theorem~\ref{thm:lemma} holds for singular $\k$, then a form of our unabridged main theorem holds for semigroups of singular cardinality.

\begin{proof}[Proof of Theorem~\ref{thm:lemma}]
As in the proof of Theorem 6.1 in \cite{B&K}, we build an increasing sequence of filters $\seq{\F_\mu}{\mu < 2^\k}$ by transfinite recursion, starting with $\F_0 = \mathcal{FR}$. In the end, $\F = \bigcup_{\mu < 2^\k}\F_\mu$ will satisfy the conclusions of the theorem.

To facilitate the construction, we begin with a matrix $\set{\mathcal M^i}{i \in I_0}$ that is an $(\F_0,\G)$-independent $\widehat{\ \ }$step-family matrix over $I_0 = 2^\k$. Along with the $\F_\mu$, we also obtain a decreasing sequence $\seq{I_\mu}{\mu < 2^\k}$ such that, at stage $\mu$, $\set{\mathcal M^i}{i \in I_\mu}$ is an $(\F_\mu,\G)$-independent $\widehat{\ \ }$step-family matrix over $I_\mu$.

Step $\mu$ of the recursion looks different depending on whether the integer part of $\mu$ is even or odd (the \emph{integer part} of $\mu$ being the unique $n$ such that $\mu = \lambda+n$ for some limit ordinal $\lambda$). At the even steps, we ensure that our filter will be an ultrafilter on $\G$, and at the odd steps, we ensure that our filter's Stone dual will be a $\widehat{\ \ }$set. To do this, define $B_\mu, C^r_\mu \in \G$, for $\mu < 2^\k$ and $r \in [\k]^{<\k}$, so that
\begin{itemize}
\item $\G = \set{B_\mu}{\mu < 2^\k \text{ and } \mu \text{ is even}}$.
\item Each $\seq{C^\mu_r}{r \in [\k]^{<\w}}$ is a monotone sequence of elements of $\G$, and every such sequence appears as $\seq{C^\mu_r}{r \in [\k]^{<\w}}$ for $2^\k$ distinct odd values of $\mu$.
\end{itemize}

The following list of conditions needs to be satisfied at every stage of our recursion. Note that a similar list is given in \cite{B&K}, but we have had to modify their condition $(5)$ and add a new condition $(7)$.
\begin{enumerate}
\item If $\mu < \nu$, then $\F_\mu \sub \F_\nu$ and $I_\mu \supseteq I_\nu$.
\item For limit $\nu$, $\F_\nu = \bigcup_{\mu < \nu}\F_\mu$ and $I_\nu = \bigcap_{\mu < \nu}I_\mu$.
\item Each $I_{\mu} \setminus I_{\mu+1}$ is finite.
\item $\set{\mathcal M^i}{i \in I_\mu}$ is an $(\F_\mu,\G)$-independent $\widehat{\ \ }$step-family matrix over $I_\mu$.
\item If $\mu$ is even, then either $B_\mu \in \F_{\mu+1}$ or there is some $C \in \F_{\mu+1}$ such that $C \cap B_\mu \notin \G$.
\item If $\mu$ is odd and each $C^\mu_r \in \F_\mu$, then there are $D^\mu_\a \in \F_{\mu+1}$ for $\a < \k^+$ such that $D^\mu_{\boxed{p}} \subst C^\mu_{\widehat p}$ for all $p \in [\k^+]^{<\w}$.
\item $\F_\mu$ is a filter on $\G$.
\end{enumerate}

It is clear that if conditions $(5)$ and $(7)$ are satisfied for every $\mu$, then $\F$ will be an ultrafilter on $\G$. It is also clear that if condition $(6)$ is satisfied at every stage then $\tilde \F$ will be a $\widehat{\ \ }$set. To finish the proof of the theorem, we just need to check that it is actually possible to carry out such a construction.

For limit steps the construction is prescribed by $(2)$, and it is clear that the other conditions are not violated at limits. Given $\mathcal E \sub \mathcal{P}(S)$, let $\gen{\mathcal E}$ denote the filter generated by $\mathcal E$.

For the odd step, we may assume that each $C^\mu_r \in \F_\mu$, since otherwise we can just put $\F_{\mu+1} = \F_\mu$ and $I_{\mu+1} = I_\mu$. Pick any $i \in I_\mu$ and set $I_{\mu+1} = I_\mu \setminus \{i\}$. For each $\a < \k^+$ define
$$D_\a = A_\a^i \cap \bigcup \set{C_s^i \cap E_s^i}{s \in [\k]^{<\w}}$$
and let $\F_{\mu+1} = \gen{\F_\mu \cup \set{D_\a}{\a < \k^+}}$. This is the same as in \cite{B&K}, and it is verified there that conditions $(1)-(4)$ and $(6)$ are satisfied at $\F_{\mu+1}$. Condition $(5)$ is vacuously satisfied at the odd step, so we just need to check that condition $(7)$ is satisfied. For this, note that every element of $\F_{\mu+1}$ contains an element of the form $D_{\boxed{p}} \cap C$ for some $p \in [\k^+]^{<\w}$ and $C \in \F_\mu$. But
\begin{align*}
D_{\boxed{p}} \cap C & = A_\boxed{p}^i \cap \bigcup \set{C_s^i \cap E_s^i}{s \in [\k]^{<\w}} \cap C \\
& \sub A_\boxed{p}^i \cap E_{\widehat p}^i \cap (C^i_{\widehat p} \cap C),
\end{align*}
 and the latter set is in $\G$ because $C^i_{\widehat{p}} \cap C \in \F_\mu$ and condition $(4)$ is satisfied at $\mu$.

For the even step, we have three cases.

\emph{Case I: } Suppose $\gen{\F_\mu \cup \{B_\mu\}}$ is not a filter on $\G$. In this case, put $\F_{\mu+1} = \F_\mu$ and $I_{\mu+1} = I_\mu$. Since $\gen{\F_\mu \cup \{B_\mu\}}$ is not a filter on $\G$, there is some $C \in \F_\mu$ with $B_\mu \cap C \notin \G$, so that $(5)$ is satisfied. It is easy to see that the other conditions are also satisfied.

\emph{Case II: } Suppose $\gen{\F_\mu \cup \{B_\mu\}}$ is a filter on $\G$ and $\set{\mathcal M^i}{i \in I}$ is an $(\gen{\F_\mu \cup \{B_\mu\}},\G)$-independent $\widehat{\ \ }$step-family matrix over $I_{\mu}$. In this case, put $\F_{\mu+1} = \gen{\F_{\mu} \cup \{B_{\mu}\}}$ and $I_{\mu+1} = I_{\mu}$. Again, it is easy to see that all of our conditions are satisfied.

\emph{Case III: } Suppose $\gen{\F_{\mu} \cup \{B_{\mu}\}}$ is a filter on $\G$, but $\set{\mathcal M^i}{i \in I}$ is not an $(\gen{\F_{\mu} \cup \{B_{\mu}\}},\G)$-independent $\widehat{\ \ }$step-family matrix over $I_{\mu}$. Then there are $p_0,\dots,p_n \in [\k^+]^{<\w}$, $s_0,\dots,s_n \in [\k]^{<\w}$ with each $\widehat p_k \sub s_k$, and $i_0,\dots,i_n \in I$ with all of the $i_k$ distinct, such that
$$\left( A^{i_0}_{\boxed{p_0}} \cap E^{i_0}_{s_0} \right) \cap \dots \cap \left( A^{i_n}_{\boxed{p_n}} \cap E^{i_n}_{s_n} \right) \cap (B_{\mu} \cap C_0) \notin \G$$
for some $C_0 \in \F$. In this case, put
$$\F_{\mu+1} = \gen{\F_{\mu} \cup \{ A^{i_0}_{\boxed{p_0}},\dots,A^{i_n}_{\boxed{p_n}} \} \cup \{ E^{i_0}_{s_0},\dots,E^{i_n}_{s_n} \}}$$
and $I_{\mu+1} = I_{\mu} \setminus \{i_0,\dots,i_n\}$. Letting
$$C = \left( A^{i_0}_{\boxed{p_0}} \cap E^{i_0}_{s_0} \right) \cap \dots \cap \left( A^{i_n}_{\boxed{p_n}} \cap E^{i_n}_{s_n} \right) \cap C_0,$$
we have $C \in \F_{\mu+1}$ and $B_{\mu} \cap C \notin \G$, so that condition $(5)$ is satisfied. Conditions $(1)-(3)$ and $(6)$ are clearly satisfied. Condition $(7)$ is satisfied because condition $(4)$ is satisfied at stage $\mu$. It remains to check that $(4)$ is satisfied at stage $\mu+1$.

Let $j_0,\dots,j_m \in I_{\mu+1}$ and let $C \in \F_{\mu+1}$. Given how $\F_{\mu+1}$ is defined,
$$C \supseteq \bigcap_{\ell \leq n} \left( A^{i_\ell}_{\boxed{p_\ell}} \cap E^{i_\ell}_{s_\ell} \right) \cap C_0$$
for some $C_0 \in \F_{\mu}$. Thus
$$\bigcap_{k \leq m} \left( A^{j_k}_{\boxed{p_k}} \cap E^{j_k}_{s_k} \right) \cap C \supseteq \bigcap_{k \leq m} \left( A^{j_k}_{\boxed{p_k}} \cap E^{j_k}_{s_k} \right) \cap \bigcap_{\ell \leq n} \left( A^{i_\ell}_{\boxed{p_\ell}} \cap E^{i_\ell}_{s_\ell} \right) \cap C_0,$$
and this is in $\G$ because $\set{\mathcal M^i}{i \in I_\mu}$ is an $(\F_{\mu},\G)$-independent $\widehat{\ \ }$step-family matrix over $I_{\mu}$.
\end{proof}

\section{The main theorem (unabridged)}\label{sec:general}

In this section we prove the unabridged version of our main theorem. Most of the work is already done: the main algebraic ingredients were developed in Sections \ref{sec:prelims} and \ref{sec:countable}, and the main set-theoretic ingredients were developed in Section~\ref{sec:lemma}. In order to make the results of Section~\ref{sec:lemma} applicable to $U(S)$, we have the following lemma:

\begin{lemma}\label{lem:thickislarge}
Let $S$ be an uncountable semigroup with $\card{S} = \k$.
\begin{enumerate}
\item If $\k$ is regular and $S$ is very weakly cancellative, $\thick_S$ is a large semifilter.
\item If $\k$ is singular and $S$ is $\mu$-weakly cancellative for some $\mu < \k$, $\thick_S$ is a large semifilter.
\end{enumerate}
\end{lemma}
\begin{proof}
The proofs of $(1)$ and $(2)$ are essentially identical except at one point. So we will prove both simultaneously, but at one point will consider two cases according to whether $\k$ is regular or singular.

Let $A \in \thick_S$ and $A \subst B$. By Proposition~\ref{prop:fact1}, there is some minimal left ideal $L$ with $L \sub \closure A$. By Proposition~\ref{prop:fact3}, $L \sub U(S)$. But $A \subst B$ implies $\closure A \cap U(S) \sub \closure B \cap U(S)$. Thus $L \sub \closure B$ and, by another application of Proposition~\ref{prop:fact1}, $B \in \thick_S$. This checks the first half of the definition of a large semifilter.

To check the other half, we build a partition $\set{X_\a}{\a < \k}$ of $S$ by recursion. Write $S = \set{s_\a}{\a < \k}$, and for each $\a < \k$ let $I_\a = \set{s_\xi}{\xi < \a}$. Our construction will ensure that, for every $\a$, $\card{X_\a} \leq \card{\a+1} < \k$ and there is some $t_\a \in S$ such that $I_\a \cdot t_\a \sub X_\a$.

Fix $\a < \k$ and assume that $X_\b$ has already been defined for every $\b < \a$.

We claim that there is some $t_\a \in S$ such that
$$(I_\a \cdot t_\a) \cap \bigcup_{\b < \a}X_b = \0.$$
Suppose this is not the case. Let $\lambda = \card{\a} \cdot \aleph_0 < \k$. Since $\card{\a} < \lambda$ and $\card{X_\b} = \card{\b+1} \leq \lambda$ for every $\b$, we have $\theta = \card{\bigcup_{\b < \a}X_\b} \leq \lambda$. Write $\bigcup_{\b < \a}X_\b = \set{r_\zeta}{\zeta < \theta}$. Since we are assuming that
$$(I_\a \cdot s) \cap \bigcup_{\b < \a}X_b = (\set{s_\xi}{\xi < \a} \cdot s) \cap \set{r_\zeta}{\zeta < \theta} \neq \0$$
for every $s \in S$, we can choose for every $s \in S$ some pair $p(s) = (\xi(s),\zeta(s)) \in \a \times \theta$ such that $s_{\xi(s)} \cdot s = r_{\zeta(s)}$. Note that $\card{\a \times \theta} \leq \lambda < \k$. We now have two cases.

If $\k$ is regular, then by the pigeonhole principle there is some $A \sub \k$ with $\card{A} = \k$ such that $p(s)$ is the same for every $s \in A$. This contradicts very weak cancellativity, which states that the equation $s_{\xi(s)} \cdot x = r_{\zeta(s)}$ must have fewer than $\k$ solutions.

If $\k$ is singular, then there is some $A \sub \k$ with $\card{A} \geq \mu$ such that $p(s)$ is the same for every $s \in A$ (otherwise $\k \leq \lambda \cdot \mu$, since $S = \bigcup_{(a_1,a_2) \in A \times A}p^{-1}(a_1,a_2)$, but this is absurd). This contradicts $\mu$-weak cancellativity, which states that the equation $s_{\xi(s)} \cdot x = r_{\zeta(s)}$ must have fewer than $\mu$ solutions.

In either case, there is some $t_\a \in S$ such that $(I_\a \cdot t_\a) \cap \bigcup_{\b < \a}X_b = \0$. If $s_\a \in \bigcup_{\b < \a}X_\b$, then let $X_\a = A_\a \cdot t_\a$. If $s_\a \notin \bigcup_{\b < \a}X_\b$, let $X_\a = (A_\a \cdot t_\a) \cup \{s_\a\}$. Clearly $\card{X_\a} \leq \card{I_\a}+1 \leq \card{\a+1}$.

By construction, $\set{X_\a}{\a < \k}$ is a partition of $S$ with $\card{X_\a} < \k$ for every $\a$. Now suppose $A \in [\k]^{\k}$. If $F$ is any finite subset of $S$, $F \sub I_\a$ for some $\a < \k$. There is some $\b \in A$ with $\b > \a$, and $F \cdot t_\b \sub I_\b \cdot t_\b \sub X_\b$. Since $F$ was arbitrary, $\bigcup_{\b \in A}X_\b$ is $S$-thick. This shows that $\thick_S$ satisfies the second part of the definition of a large semifilter.
\end{proof}

We can now prove our unabridged main theorem:

\begin{theorem}\label{thm:Psets2}
Suppose $S$ is very weakly cancellative and $\card{S} = \k$ for some regular cardinal $\k$. There is a minimal left ideal of $\b S$ that is a weak $P_{\k^+}$-set in $U(S)$.
\end{theorem}
\begin{proof}
This follows immediately from Theorem~\ref{thm:lemma}, Proposition~\ref{prop:thicksets}, and Lemma~\ref{lem:thickislarge}.
\end{proof}

\begin{theorem}\label{thm:betaS2}
Suppose $S$ is left-cancellative and very weakly right cancellative, and $\card{S} = \k$ is a regular cardinal. There is a minimal left ideal of $U(S)$ that is also prime.
\end{theorem}
\begin{proof}
The proof is similar to that of Theorem~\ref{thm:betaS}.

By Theorem~\ref{thm:Psets2}, there is some $L \sub U(S)$ that is both a weak $P_{\k^+}$-set and a minimal left ideal. By Lemma~\ref{lem:exercise}, if $p \notin L$ then $(S \cdot p) \cap L = \0$. Since $L$ is a weak $P_{\k^+}$-set, $\b S \cdot p = \closure S^{\b S} \cdot p = \closure{S \cdot p}^{\b S}$ is disjoint from $L$. Therefore, if $p \notin L$ then $q \cdot p \notin L$. In other words, $L$ is prime.
\end{proof}

Theorem~\ref{thm:Psets2} is part $(1)$ of our main theorem, and Theorem~\ref{thm:betaS2} is part $(2)$. Proposition~\ref{prop:2implies3} shows that part $(2)$ implies part $(3)$ and completes the proof.

Notice that we have nearly proved a version of our unabridged main theorem for singular $\k$. If we replace the assumption of very weak cancellativity with the assumption that $S$ is $\mu$-weakly cancellative for some $\mu < \card{S}$, then the whole argument remains valid, except for Lemma~\ref{lem:matrix}. At no other point do we use the regularity of $\k$. Thus a positive answer to Question~\ref{question:matrix} would automatically give a version of our main theorem for semigroups of singular cardinality.


\begin{thebibliography}{99}
\bibitem{B&K} J. Baker and K. Kunen, ``Limits in the uniform ultrafilters,'' \emph{Transactions of the American Mathematical Society} \textbf{353} (2001), pp. 4083-4093.
\bibitem{BK2} J. Baker and K. Kunen, ``Matrices and ultrafilters,'' in \emph{Recent Progress in General Topology II}, eds. M. Hu\v{s}ek and J. van Mill, North-Holland, Amsterdam, 2002, pp. 59-81.
\bibitem{BHM} V. Bergelson, N. Hindman, and R. McCutcheon, ``Notions of size and combinatorial properties of quotient sets in semigroups,'' \emph{Topology Proceedings} \textbf{23} (1998), pp. 23-60.
\bibitem{Brn} W. R. Brian, ``$P$-sets and minimal right ideals in $\N^*$,'' \emph{Fundamenta Mathematicae} \textbf{229} (2015), pp. 277-293.
\bibitem{B&V} W. R. Brian and J. Verner, ``$G_\delta$ semifilters and $\w^*$,'' preprint available at \texttt{arxiv.org/abs/1503.06092}.
\bibitem{H&S} N. Hindman and D. Strauss, \emph{Algebra in the Stone-\v{C}ech compactification}, De Gruyter, Berlin, 1998.
\bibitem{HSZ} N. Hindman, D. Strauss, and Y. Zelenyuk, ``Longer chains of idempotents in $\b G$,'' \emph{Fundamenta Mathematicae} \textbf{220} (2013), pp. 243-261.
\bibitem{Kun} K. Kunen, ``Weak $P$-points in $N^*$,'' vol. Proc. J\'anos Bolyai Soc. Colloq. on Topology, vol. 23, Budapest, 1978, pp. 741-749.
\bibitem{Shl} S. Shelah, \emph{Proper and Improper Forcing}, Perspectives in Mathematical Logic, Springer, Berlin, 1998.
\bibitem{Zln} Y. Zelenyuk, ``Principal left ideals of $\b G$ may be both minimal and maximal,'' \emph{Bulletin of the London Mathematical Society} \textbf{45} (2013), pp. 613-617.
\bibitem{YZ2} Y. Zelenyuk, ``Left maximal idempotents in $G^*$,'' \emph{Advances in Mathematics} \textbf{262} (2014), pp. 593-603.
\end{thebibliography}
\end{document}